\documentclass[11pt,reqno]{amsart}

\usepackage{amsmath,amsthm,amssymb,amsfonts}
\usepackage{graphicx}

\newtheorem{theorem}{Theorem}[section]
\newtheorem{corollary}[theorem]{Corollary}
\newtheorem{lemma}[theorem]{Lemma}

\newtheorem{definition}[theorem]{Definition}

\numberwithin{equation}{section}

\font\script=rsfs10 at 11pt
\def\H{{\mbox{\script H}\;}}
\def\N{\mathbb N}
\def\R{\mathbb R}
\def\S{\mathbb S}
\def\eps{\varepsilon}

\def\angle#1#2#3{#1\widehat{#2}#3}
\usepackage{yhmath}
\def\arc#1{\wideparen{#1}}

\newcounter{mt}
\def\maintheorem#1#2#3{\par \medskip \noindent {\bf Theorem~\mref{#1}}~(#2).~{\it #3}\par}
\def\mref#1{\Alph{#1}}
\def\maintheoremdeclaration#1{\stepcounter{mt}\newcounter{#1}\setcounter{#1}{\arabic{mt}}}
\maintheoremdeclaration{main}
\title{Regularity of isoperimetric sets in $\R^2$ with density}

\author{E. Cinti}
\author{A. Pratelli}

\begin{document}

\begin{abstract}
We consider the isoperimetric problem in $\R^n$ with density for the planar case $n=2$. We show that, if the density is ${\rm C}^{0,\alpha}$, then the boundary of any isoperimetric is of class ${\rm C}^{1,\frac \alpha{3-2\alpha}}$. This improves the previously known regularity.
\end{abstract}

\maketitle

\section{Introduction}

The isoperimetric problem in $\R^n$ with density is a classical problem, which has received much attention in the last decade. The idea is quite simple: given a l.s.c. function $f:\R^n\to (0,+\infty)$, usually called ``density'', for any set $E\subseteq \R^n$ we define the volume $V_f(E)$ and the perimeter $P_f(E)$ as
\begin{align*}
V_f(E) := \int_E f(x)\, dx\,, &&
P_f(E) := \int_{\partial^* E} f(x)\, d\H^{n-1}(x)\,,
\end{align*}
where the subscript reminds the fact that volume and perimeter are done with respect to $f$, and where $\partial^* E$ is the reduced boundary of $E$ (to read this paper there is no need to know what the reduced boundary is, since under our assumptions this will always coincide with the usual topological boundary $\partial E$). The literature on this problem is huge, a short and incomplete list is~\cite{All2,BigReg,Bom,Morgan2005,CMV,CP1,DHHT,FMP2,MP,RCBM}).\par

The main questions about the isoperimetric problem with density are of course the existence and regularity of isoperimetric sets; concerning the existence, now a lot is known (see for instance~\cite{MP,DFP,GC}). Since in this paper we are dealing with the regularity, let us briefly recall the main known results. The first, very important one, can be found in~\cite[Proposition~3.5 and Corollary~3.8]{Morgan2003} (see also~\cite{All1,Alm}).
\begin{theorem}\label{veryold}
Let $f$ be a ${\rm C}^{k,\alpha}$ density on $\R^n$, with $k\geq 1$. Then the boundary of any isoperimetric set is ${\rm C}^{k+1,\alpha}$, except for a singular set of Hausdorff dimension at most $n-8$. If $f$ is just Lipschitz and $n=2$, then the boundary is of class ${\rm C}^{1,1}$.
\end{theorem}

It is important to point out that the above result requires at least a ${\rm C}^{1,\alpha}$ regularity for $f$ (or Lipschitz in the $2$-dimensional case). The reason is rather simple: namely, most of the standard techniques to get the regularity make sense only if $f$ is at least Lipschitz (for a discussion on this fact, see~\cite[Section~5]{CP1}). In particular, the following observation can be particularly insightful. In order to get regularity of an isoperimetric set $E$, a standard idea is to build some ``competitor'' $F$, which behaves better than $E$ where the boundary of $E$ is not regular enough; however, in order to get some contradiction, we must ensure that $F$ has the same volume as $E$, since otherwise the isoperimetric property of $E$ cannot be used. On the other hand, it can be complicate to build the set $F$ taking its volume into account, since while defining $F$ one is interested in its perimeter. As a consequence, it is extremely useful to have the so-called ``$\eps-\eps$ property'', which roughly speaking says the following: it is always possible to modify the volume of a set $F$ of a small quantity $\eps$, increasing its perimeter of at most $C|\eps|$ for some fixed constant $C$; if this is the case, one can then ``adjust'' the volume of $F$ so that it coincides with the one of $E$. This property was already discussed by Allard, Almgren and Bombieri since the 1970's (see for instance~\cite{All1, All2, Alm, BigReg, Bom}), and it has been widely used in most of the papers about the regularity in this context since then. Unfortunately, while the $\eps-\eps$ property is rather simple to establish when $f$ is at least Lipschitz, it is false if $f$ is not Lipschitz.\par

To get anyhow some regularity for the isoperimetric sets in case of low regularity of $f$, in the recent paper~\cite{CP1} we have introduced and proved a weaker property, called the ``$\eps-\eps^\beta$ property'', which basically says that the volume of a set can be modified of $\eps$, increasing the perimeter of at most $C|\eps|^\beta$. Since we will use this property in the present paper, we give here the result (the actual result proved in~\cite{CP1} is more general, but we prefer to claim here the simpler version that we are going to need).

\begin{theorem}[\cite{CP1}, Theorem~B]\label{oldepsepsbeta}
Let $E\subseteq\R^n$ be a set of locally finite perimeter, and $f$ an $\alpha$-H\"older density for some $0<\alpha\leq 1$. Then, for every ball $B$ with nonempty intersection with $\partial^* E$, there exist two constants $\bar \eps,\, C>0$ such that, for every $|\eps| < \bar\eps$, there is a set $\widetilde E$ satisfying
\begin{align}\label{propepepbe}
\widetilde E\Delta E\subset\subset B\,, && V_f(\widetilde E) = V_f (E) + \eps\,, && P_f(\widetilde E) \leq P_f(E) + C |\eps|^\beta\,,
\end{align}
where $\beta=\beta(\alpha,n)$ is defined by
\[
\beta=\beta(\alpha,n) := \frac{\alpha+(n-1)(1-\alpha)}{\alpha+n(1-\alpha)}\,,
\]
so for $n=2$ it is $\beta=\frac1{2-\alpha}$.
\end{theorem}

By using the classical regularity results (see for instance~\cite{Amb,Tam}), and modifying the arguments in order to make use of the $\eps-\eps^\beta$ property instead of the $\eps-\eps$ one (which is false), we then obtained the following regularity result (see~\cite{Tam} for a definition of porosity).

\begin{theorem}[\cite{CP1}, Theorem~5.7]\label{old}
Let $E$ be an isoperimetric set in $\R^n$ with a density $f\in {\rm C}^{0,\alpha}$, with $0<\alpha\leq 1$. Then $\partial^* E=\partial E$ is of class $C^{1,\frac{\alpha}{2n(1-\alpha)+2\alpha}}$. In particular, if $n=2$ then $\partial E$ is ${\rm C}^{1,\frac{\alpha}{4-2\alpha}}$. If $f$ is only bounded above and below, then it is still true that $\partial^* E=\partial E$, and moreover $E$ is porous.
\end{theorem}

The main result of the present paper is the following stronger regularity result for the bi-dimensional case.

\maintheorem{main}{Regularity of isoperimetric sets}{Let $f:\R^2\to (0,+\infty)$ be a ${\rm C}^{0,\alpha}$ density, for some $0<\alpha\leq 1$. Then, every isoperimetric set $E$ has a boundary of class ${\rm C}^{1,\frac\alpha{3-2\alpha}}$.}

It is worthy observing that the regularity obtained above is still less than the ${\rm C}^{1,\alpha}$ regularity that one could expect just by looking at Theorem~\ref{veryold}, but it is better than the previously known regularity given by Theorem~\ref{old}. In particular, notice that there is a substantial improvement between ${\rm C}^{1,\frac \alpha{4-2\alpha}}$ and ${\rm C}^{1,\frac \alpha{3-2\alpha}}$. Indeed, the second exponent is not just merely bigger than the first one, there is also a much deeper difference: namely, when $\alpha$ goes to $1$, the first exponent goes to $1/2$, while the second goes to $1$. In particular, our Theorem~\mref{main} gives also a proof that for $f$ Lipschitz an isoperimetric set is ${\rm C}^{1,1}$, as stated in Theorem~\ref{veryold}.

We conclude this introduction with a couple of remarks. First of all, the fact that the regularity of the ``old'' Theorem~\ref{old} does never exceed ${\rm C}^{1,\frac 12}$ is not strange, since this is the best regularity that can be obtained via the classical methods, not really using the (non-local) isoperimetric property of a set $E$, but only the weaker (local) fact that $E$ is an $\omega$-minimizer of the perimeter. To get anything better than ${\rm C}^{1,\frac 12}$, one has really to use the non-locality of the fact that $E$ is an isoperimetric set, as we do in the present paper by using the non-local $\eps-\eps^\beta$ property. The second remark is about the sharp regularity exponent that one can obtain for an isoperimetric set with a ${\rm C}^{0,\alpha}$ density: we do not believe that our exponent of Theorem~\mref{main} is sharp, but we are also not sure whether it is possible to reach the ${\rm C}^{1,\alpha}$ regularity, similarly to what happens for $k\geq 1$ in the classical case.

\subsection{Notation}
Let us briefly present the notation of the present paper. The density will always be denoted by $f:\R^2\to (0,+\infty)$; keep in mind that, since we want to prove Theorem~\mref{main}, the function $f$ will always be at least continuous. For any set $E\subseteq \R^2$, we call $V_f(E)$ and $P_f(E)$ its volume and perimeter. For any $z\in\R^2$ and $\rho>0$, we call $B_\rho(z)$ the ball centered at $z$ with radius $\rho$. Given two points $x,\,y\in\R^2$, we denote by $xy,\, \ell(xy)$, and $\ell_f(xy)$ the segment connecting the two points, its Euclidean length, and its length with respect to the density $f$ (that is, $\int_{xy} f(t) dt$). Given three points $a,\,b,\, c$, we will denote by $\angle abc$ the angle between the segments $ab$ and $bc$. Let now $E\subseteq \R^2$ be an isoperimetric set; then, by Theorem~\ref{old} we already know that the boundary of $E$ is of class ${\rm C}^{1,\frac \alpha{4-2\alpha}}$, hence in particular it is locally Lipschitz. As a consequence, for any two points $x,\,y$ which belong to the same connected component of $\partial E$ and which are very close to each other (with respect to the diameter of this connected component), the shortest curve in $\partial E$ connecting $x$ and $y$ is well-defined. We denote by $\arc{xy}$ this curve, and again by $\ell(\arc{xy})$ and $\ell_f(\arc{xy})$ we denote its Euclidean length, and its length with respect to the density $f$. The letter $C$ is always used to denote a large constant, which can increase from line to line, while $M$ is a fixed constant, coming from the $\alpha$-H\"older property.

\section{Proof of the main result}

This section is devoted to the proof of our main result, Theorem~\mref{main}. Most of the proof consists in studying the situation around few given points, so let us fix some useful particular notation; Figure~\ref{figure} helps with the names of the points. Let us fix an isoperimetric set $E$ and a point $z$ on $\partial E$. Let $\rho\ll 1$ be a very small constant, much smaller than the length of the connected component $\gamma\subseteq\partial E$ containing $z$, and of the distance between $z$ and the other connected components of $\partial E$ (if any). Since we know that $\gamma$ is a ${\rm C}^{1,\frac{\alpha}{4-2\alpha}}$ curve, we can fix arbitrarily an orientation on it; hence, let us define four points $x,\, \bar x,\,y,\,\bar y$ in $\partial B_\rho(z)$ as follows: among the points of $\gamma$ which belong to $\partial B_\rho(z)$ and which are \emph{before} $z$, we call $\bar x$ the closest one to $z$, and $x$ the farthest one (in the sense of the paremeterization of $\gamma$). We define analogously $\bar y$ and $y$ \emph{after} $z$: of course, $\bar x$ and $x$ may coincide, as well as $y$ and $\bar y$. Moreover, we call $\delta$ the angle $\angle zxy$, and
\[
l:= \ell(\arc{x\bar x})+ \ell(\arc{y\bar y})\,.
\]
Finally, we introduce the following set $F$, which we will use as a competitor to $E$ (after adjusting its area).
\begin{figure}[thbp]
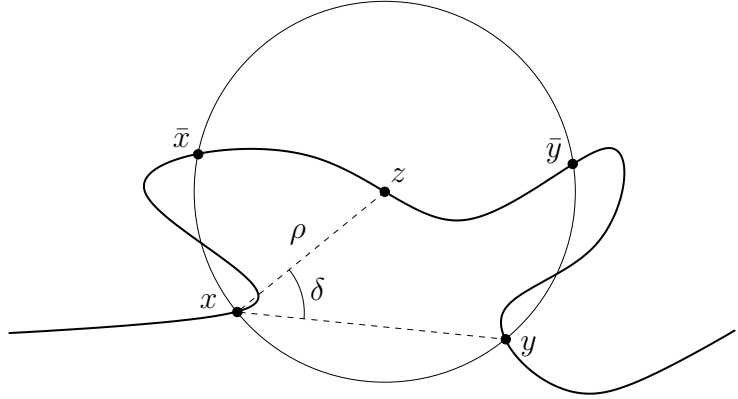
\caption{Position of the points for Lemma~\ref{generic}.}\label{figure}
\end{figure}
\begin{definition}\label{defF}
With the above notation, we let $F\subseteq \R^2$ be the bounded set whose boundary is $\partial F = \partial E \setminus \arc{xy} \cup xy$.
\end{definition}
It is to be noticed that the above definition makes sense: indeed, by~\cite[Theorem~1.1]{CP1} we already know that $E$ is bounded, and by construction the segment $xy$ cannot intersect any point of $\partial E \setminus \arc{xy}$. In particular, all the connected components of $E$ whose boundary is not $\gamma$ belong also to $F$; instead, the connected component of $E$ with $\gamma$ as boundary has been slightly changed near $z$. Recalling that $f$ is real-valued and strictly positive, as well as $\alpha$-H\"older, we can find a constant $M$ such that
\begin{align}\label{genass}
\frac 1M \leq f(p)\leq M\,, && |f(p)-f(q)| \leq M |p-q|^\alpha
\end{align}
for every $p,\, q$ in some neighborhood of $E$, containing all the points that we are going to use in our argument: this is not a problem, since all our arguments will be local. Let us now give the first easy estimates about the above quantities.
\begin{lemma}\label{generic}
With the above notation, one has
\begin{align}\label{easyfirst}
l \leq C\rho^{\frac 2{2-\alpha}}\,, &&
\delta \leq C\rho^{\frac \alpha{4-2\alpha}}\,,
\end{align}
and moreover
\begin{equation}\label{esticirc}
\ell_f(\arc{xy}) - \ell_f(xy) \ll \ell(xy)\,.
\end{equation}
\end{lemma}
\begin{proof}
First of all, let us consider a point $p\in E\Delta F$ in the symmetric difference between $E$ and $F$; by construction, either $p$ belongs to the ball $B_\rho(z)$, or it has a distance at most $l/2$ from that ball. As a consequence, $p$ has distance at most $\rho+l/2$ from $z$, and this gives
\begin{equation}\label{put0}
|V_f(E)-V_f(F)| \leq V_f(E\Delta F) \leq M\big(\pi (\rho+l/2)^2\big) \leq 2M\pi (\rho^2 + l^2)\,.
\end{equation}
Let us now apply Theorem~\ref{oldepsepsbeta} to the set $E$, with a ball $B$ intersecting $\partial E$ far away from the point $z$. We get a constant $\bar\eps$ and of course, up to have taken $\rho$ small enough, we can assume that $|\eps|\leq \bar \eps$, being $\eps = V_f(E)-V_f(F)$. Then, Theorem~\ref{oldepsepsbeta} provides us with a set $\widetilde E$ satisfying~(\ref{propepepbe}); as a consequence, if we define $\widetilde F= F \setminus B \cup (B\cap \widetilde E)$, we get $V_f(\widetilde F) = V_f(F) + V_f(\widetilde E) - V_f(E) = V_f(E)$, and then, since $E$ is isoperimetric, by~(\ref{put0}) we get
\[
P_f(E)\leq P_f(\widetilde F) = P_f(F) + P_f(\widetilde E) - P_f(E) 
\leq P_f(F)+ C \big(\rho^2+l^2\big)^{\frac 1{2-\alpha}}\,,
\]
which implies
\begin{equation}\label{put1}
\ell_f(\arc{xy}) - \ell_f(xy)= P_f(E) -P_f(F) \leq C\big(\rho^2+l^2\big)^{\frac 1{2-\alpha}}\,.
\end{equation}
Let us now evaluate the term $\ell_f(\arc{xy}) - \ell_f(xy)$: if we call $f_{\rm min}$ and $f_{\rm max}$ the minimum and the maximum of $f$ inside $B_z(\rho)$, by~(\ref{genass}) we have
\begin{align*}
f_{\rm max} \leq f_{\rm min} + 2M\rho^\alpha\,, && f_{\rm min}\geq \frac 1M\,,
\end{align*}
and then
\[\begin{split}
\ell_f(\arc{xy})-\ell_f(xy)
&= \ell_f(\arc{\overline x \,\overline y})+\ell_f(\arc{x\bar x}) +\ell_f(\arc{y\bar y})-\ell_f(xy)
\geq  2\rho f_{\rm min}+\frac l M - 2\rho \cos(\delta) f_{\rm max} \\
&\geq 2\rho f_{\rm min}+\frac l M - 2\rho \cos(\delta) \big(f_{\rm min} + 2M\rho^\alpha\big)
\geq 2\,\frac{1-\cos\delta}M\, \rho +\frac l M - 4M \rho^{\alpha+1}\,.
\end{split}\]
Inserting this estimate in~(\ref{put1}), we get then
\[
2\,\frac{1-\cos\delta}M\, \rho +\frac l M  \leq C\big(\rho^2+l^2\big)^{\frac 1{2-\alpha}}+4M \rho^{\alpha+1}\,.
\]
Since 
\begin{align*}
\frac 2{2-\alpha} > 1\,, && \alpha+1 \geq \frac 2{2-\alpha}\,,
\end{align*}
we immediately derive first that $\delta$ is very small, so that $1-\cos\delta \geq \delta^2/3$, and then
\[
\delta^2 \rho + l \leq C \rho^{\frac 2{2-\alpha}}\,.
\]
This gives the validity of both the inequalities in~(\ref{easyfirst}). Finally, (\ref{put1}) together with~(\ref{easyfirst}) and the fact that, since $\delta\ll 1$, one has $\ell(xy)\approx 2\rho$, gives~(\ref{esticirc}).
\end{proof}
\begin{corollary}
For any two points $a,\,b\in \gamma$ sufficiently close to each other, one always has
\begin{align}\label{esticirca}
\ell_f(\arc{ab}) - \ell_f(ab) \ll \ell_f(ab)\,, &&
\ell(\arc{ab}) - \ell(ab) \ll \ell(ab)\,.
\end{align}
\end{corollary}
\begin{proof}
Let $z\in \arc{ab}$ be a point such that $\rho=\ell(az)=\ell(zb)$, and let us call $x,\,\bar x,\, y$ and $\bar y$ as before. Of course, $a\in\arc{x\bar x}$ and $b\in \arc{\overline y y}$, so that $\ell(a\bar x)+\ell(b\bar y)\leq l$. Since  by~(\ref{easyfirst}) we have $l\ll \rho$, we get $\angle az{\bar x}\ll 1$ and $\angle bz{\bar y}\ll 1$, as well as $\ell_f(ab)\approx \ell_f(xy)$. Hence, (\ref{esticirc}) gives us
\[
\ell_f(\arc{ab}) - \ell_f(ab)  \leq 
\ell_f(\arc{xy}) - \ell_f(xy)+\ell_f(xy) -\ell_f(ab)
\ll \ell(ab)\,,
\]
and then the first estimate in~(\ref{esticirca}) is established. The second one immediately follows, just thanks to the continuity of $f$.
\end{proof}

An argument similar to the one proving Lemma~\ref{generic} gives then the following estimate.
\begin{lemma}\label{applynow}
Given any two sufficiently close points $r,\, s\in \gamma$, one has
\begin{equation}\label{toprove}
\ell_f(\arc{rs}) - \ell_f(rs) \geq -12M^5\, \ell(rs)^{\frac{2+\alpha}{2-\alpha}}\,.
\end{equation}
\end{lemma}
\begin{proof}
Let us call $t$ the maximal distance between points of the arc $\arc{rs}$ and the segment $rs$, and let us denote $2d=\ell(rs)$ for brevity. Of course, concerning the Euclidean distances, one has
\begin{equation}\label{estEuc}
\ell(\arc{rs}) \geq 2\sqrt{d^2+t^2}\,.
\end{equation}
On the other hand, let us call $\pi$ the projection of $\R^2$ on the line containing the segment $rs$, so that for every $a\in\arc{rs}$ one has $|a-\pi(a)|\leq t$ by definition; moreover, define the density $g:\R^2\to\R^+$ as $g(a)=f(\pi(a))$, so that by the $\alpha$-H\"older property~(\ref{genass}) of $f$ we get
\begin{equation}\label{est2}
f(a) \geq g(a) - M t^\alpha
\end{equation}
for every $a\in\arc{rs}$. Since $\pi(a)=a$ for every $a\in rs$, by~(\ref{estEuc}) and~(\ref{genass}) we can then easily evaluate
\[
\ell_g(\arc{rs}) - \ell_f(rs) = \ell_g(\arc{rs}) - \ell_g(rs) \geq \frac 2M \,\big(\sqrt{d^2+t^2}-d\big)\,.
\]
On the other hand, by~(\ref{est2}) and by~(\ref{esticirca}) we have also
\[
\ell_f(\arc{rs}) - \ell_g(\arc{rs}) \geq - M t^\alpha \ell_f(\arc{rs}) \geq - 2M t^\alpha \ell_f(rs)
\geq - 4M^2 t^\alpha d\,.
\]
Putting together the last two estimates, we obtain then
\[
\ell_f(\arc{rs}) - \ell_f(rs) = \ell_g(\arc{rs})-\ell_f(rs)+\ell_f(\arc{rs}) - \ell_g(\arc{rs})
\geq  \frac 2M \,\big(\sqrt{d^2+t^2}-d\big)- 4M^2 t^\alpha d\,.
\]
As a consequence, we can assume that $t\ll d$, since otherwise we readily get $\ell_f(\arc{rs}) - \ell_f(rs)>0$, and in this case of course~(\ref{toprove}) holds. Therefore, the last inequality can be rewritten as
\begin{equation}\label{finhe}
\ell_f(\arc{rs}) - \ell_f(rs) \geq  \frac {2t^2}{3Md}- 4M^2 t^\alpha d
= \bigg(\frac {2t^2}{3Md^2}- 4M^2 t^\alpha\bigg) \,d\,.
\end{equation}
There are then two cases: if the term between parenthesis is positive, then again~(\ref{toprove}) clearly holds. If, instead, it is negative, this implies
\[
t\leq 6M^3 d^{\frac 2{2-\alpha}}\,,
\]
and then~(\ref{finhe}) gives
\[
\ell_f(\arc{rs}) - \ell_f(rs)\geq - 4M^2 t^\alpha d \geq - 24M^5 d^{\frac{2+\alpha}{2-\alpha}}
\geq -12M^5 \ell(rs)^{\frac{2+\alpha}{2-\alpha}}\,,
\]
which is~(\ref{toprove}).
\end{proof}

We can now show a more refined estimate for $\ell_f(\arc{pq})-\ell_f(pq)$, which takes into account the maximal angle of deviation of the curve $\arc{pq}$ with respect to the segment $pq$. Let us be more precise: for every $w\in \arc{pq}$ we define $H=H(w)$ the projection of $w$ on the line containing $pq$. Moreover, we call $\theta=\theta(w)$ the angle $\angle wqp$ if $H$ is closer to $p$ than to $q$, and $\theta=\angle wpq$ otherwise, and we let $\bar\theta$ be the maximum among all the angles $\theta(w)$ for $w\in\arc{pq}$: Figure~\ref{figure2} depicts the situation.
\begin{figure}[thbp]
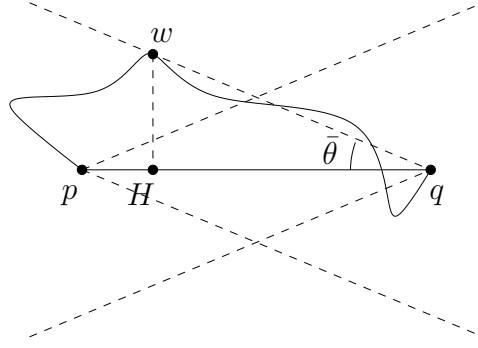
\caption{Definition of $\bar\theta$.}\label{figure2}
\end{figure}
\begin{lemma}\label{lecce}
With the above notation, and calling $\rho=\ell(pq)$, there is $C=C(\alpha,M)$ for which
\begin{equation}\label{2poss}\begin{split}
\hbox{either}\quad \bar\theta \leq C\rho^{\frac \alpha{2-\alpha}},\,\qquad \qquad &\hbox{and then}\quad \ell_f(\arc{pq})-\ell_f(pq) \geq  - C\rho^{\frac{2+\alpha}{2-\alpha}}\,, \\
\hbox{or}\quad \bar\theta \geq C\rho^{\frac \alpha{2-\alpha}},\, \qquad\qquad &\hbox{and then}\quad \ell_f(\arc{pq})-\ell_f(pq) \geq \frac \rho{12M}\,\bar\theta^2\,.
\end{split}\end{equation}
\end{lemma}
\begin{proof}
Let us fix a point $w\in\arc{pq}$ such that $\theta(w)=\bar\theta$, and let assume without loss of generality that $\bar\theta=\angle wqp$, as in Figure~\ref{figure2}. We can then apply Lemma~\ref{applynow}, first with $rs=pw$, and then with $rs=wq$, to get, also keeping in mind~(\ref{esticirca}), that
\begin{equation}\label{lecce1}\begin{split}
\ell_f(\arc{pq})&=\ell_f(\arc{pw}) + \ell_f(\arc{wq})
\geq \ell_f(pw) + \ell_f(wq)-12M^5 \Big(\ell(pw)^{\frac{2+\alpha}{2-\alpha}}+\ell(wq)^{\frac{2+\alpha}{2-\alpha}}\Big)\\
&\geq \ell_f(pw) + \ell_f(wq)-24M^5 \rho^{\frac{2+\alpha}{2-\alpha}}\,.
\end{split}\end{equation}
We have now to compare the lengths of the segments $pw$ and $qw$ with those of the segments $pH$ and $qH$. We can again argue defining the density $g$ as $g(a)=f(\pi(a))$, being $\pi$ the projection on the line containing the segment $pq$: then
\[\begin{split}
\ell_f(pw)-\ell_f(pH) &= \ell_g(pw)-\ell_g(pH) + \ell_f(pw)-\ell_g(pw)\\
&\geq \frac 1M\, \Big(\ell(pw)-\ell(pH)\Big) - M\ell_f(pw) \ell(wH)^\alpha\,,
\end{split}\]
and the analogous estimate of course works for $\ell_f(wq)-\ell_f(wH)$. Putting them together, and recalling that $\ell_f(pH)+\ell_f(qH)\geq \ell_f(pq)$, with strict inequality if $H$ is outside $pq$, we obtain
\[\begin{split}
\ell_f(pw)+\ell_f(wq)-\ell_f(pq) &\geq \frac 1M\, \Big(\ell(pw)+\ell(wq)-\ell(pq)\Big) - M\big(\ell_f(pw)+\ell_f(wq)\big) \ell(wH)^\alpha\\
&\geq \frac{\rho}{6M} \bar\theta^2 - 2M^2 \rho^{1+\alpha}\bar\theta^\alpha\,,
\end{split}\]
where we have again used that $\bar\theta\ll 1$, which comes as usual by~(\ref{esticirca}). Putting this estimate together with~(\ref{lecce1}), we get
\begin{equation}\label{lecce2}
\ell_f(\arc{pq})-\ell_f(pq) \geq \frac{\rho}{6M} \,\bar\theta^2 - 2M^2 \rho^{1+\alpha}\bar\theta^\alpha-24M^5 \rho^{\frac{2+\alpha}{2-\alpha}}\,.
\end{equation}
Now, notice that
\[
\rho^{1+\alpha}\bar\theta^\alpha \geq  \rho^{\frac{2+\alpha}{2-\alpha}} \qquad \Longleftrightarrow \qquad \bar\theta\geq \rho^{\frac \alpha{2-\alpha}}\qquad \Longleftrightarrow \qquad \rho\bar\theta^2 \geq  \rho^{1+\alpha}\bar\theta^\alpha \,.
\]
As a consequence, (\ref{lecce2}) implies the validity of both the cases in~(\ref{2poss}), up to have chosen a sufficiently large constant $C=C(M,\alpha)$.
\end{proof}

We are now ready to find a first result about the behaviour of the direction of the chords connecting points of the curve $\gamma$. Indeed, putting together Lemma~\ref{lecce} and Lemma~\ref{generic}, we get the next estimate.
\begin{lemma}\label{2.6}
Let $a,\, z$ be two points in $\gamma$ sufficiently close to each other, call $\rho=\ell(az)$, and let $w\in \arc{az}$ be a point closer to $a$ than to $z$. Then, there is $C=C(M,\alpha)$ such that
\[
\angle wza \leq C \rho^{\frac \alpha{3-2\alpha}}\,.
\]
\end{lemma}
\begin{proof}
Having a point $z\in\gamma$ and some $\rho>0$ very small, we can define the points $x,\, \bar x,\, y,\, \bar y$ as for Lemma~\ref{generic}, and by symmetry we can think $a\in \arc{x\bar x}$. We apply now twice Lemma~\ref{lecce}, once with $p_1=x$ and $q_1=z$, and once with $p_2=y$ and $q_2=z$. We have then the validity of~(\ref{2poss}) for the two angles $\bar\theta_1$ and $\bar\theta_2$; we claim that
\begin{equation}\label{finlem}
\bar\theta:= \max \big\{ \bar\theta_1,\, \bar\theta_2\big\} \leq C \rho^{\frac \alpha{3-2\alpha}}\,.
\end{equation}
Let us first see that this estimate implies the thesis. If the point $w$ is closer to $x$ than to $z$, then by definition and by~(\ref{finlem})
\[
\angle wzx = \theta_1(w) \leq \bar\theta_1\leq C \rho^{\frac \alpha{3-2\alpha}}\,.
\]
As a consequence, by~(\ref{easyfirst}) we have
\begin{equation}\label{torepeat}
\angle wza \leq \angle wzx + \angle xza
\leq C\rho^{\frac \alpha{3-2\alpha}} + 2\,\frac l\rho
\leq C\rho^{\frac \alpha{3-2\alpha}} + C \rho^{\frac \alpha{2-\alpha}}
\leq C\rho^{\frac \alpha{3-2\alpha}}\,,
\end{equation}
and the thesis is obtained. Suppose, instead, that $w$ is closer to $z$ than to $x$; since by assumption it is anyhow closer to $a$ than to $z$, and $\ell(ax)\leq l\ll \ell(wz)$ by~(\ref{easyfirst}), again using~(\ref{finlem}) we have
\[
\angle wzx \leq 2 \angle wxz = 2 \theta_1(w)\leq C\rho^{\frac \alpha{3-2\alpha}}\,,
\]
then the very same argument as in~(\ref{torepeat}) shows again the thesis. Summarizing, we have proved that~(\ref{finlem}) implies the thesis, and hence to conclude we only have to show the validity of~(\ref{finlem}).\par

We argue as in Lemma~\ref{generic}, defining the competitor set $F$ which has $\partial F=\partial E \setminus \arc{xy}\cup xz\cup zy$ as boundary. Notice that this is very similar to what we did in Definition~\ref{defF}, the only difference being that we are putting the two segments $xz$ and $zy$ instead of the segment $xy$. The very same argument that we presented after Definition~\ref{defF} still ensures that the set $F$ is well defined. As in Lemma~\ref{generic}, then, we have now to evaluate $|V_f(F)-V_f(E)|$ and $P_f(F)-P_f(E)$. Concerning the first quantity, we can get an estimate which is much better than~(\ref{put0}), thanks to the definition of $\bar\theta_1$ and $\bar\theta_2$.\par

\begin{figure}[thbp]
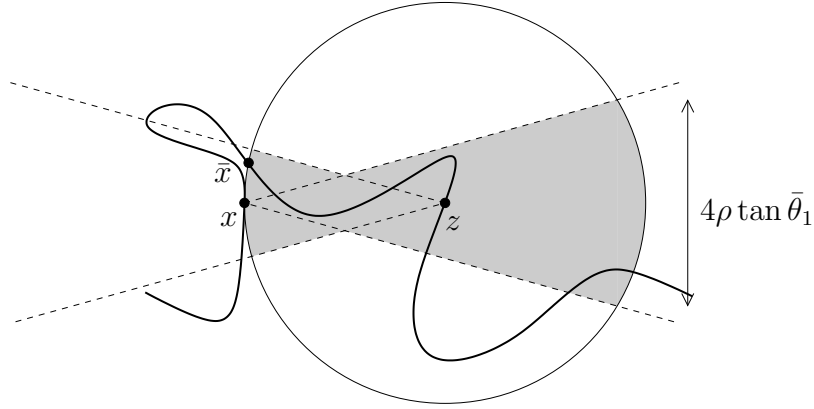
\caption{Constraint on the position of the curve $\arc{xz}$.}\label{figure3}
\end{figure}

Let us be more precise. The curve $\arc{xz}$ is composed by two pieces; concerning the first part, $\arc{x\bar x}$, by definition this remains within a distance of at most $\ell(\arc{x\bar x})$ from $x$. The second part, $\arc{\overline x z}$, is contained in the shaded region of Figure~\ref{figure3}, which is the intersection between the ball $B_\rho(z)$ and the region of the points $w$ such that $\min\{\angle wxz,\, \angle wzx \} \leq \bar\theta_1$. Repeating the same argument with $\bar\theta_2,\, y$ and $\bar y$, and recalling that $l=\ell(\arc{x\bar x})+\ell(\arc{y\bar y})$, we have
\[
|V_f(E) - V_f(F)| \leq 9M\rho^2 \big(\bar\theta_1+\bar\theta_2) + \pi M l^2 \leq 
18M\rho^2\bar\theta+ \pi M l^2\,.
\]
Let us now observe that by~(\ref{easyfirst}) one has $l^2\leq C\rho^{\frac 4{2-\alpha}}$, and then
\[
l^2 \leq \rho^2\bar\theta\quad \Longleftarrow \quad C\rho^{\frac 4{2-\alpha}} \leq \rho^2\bar\theta 
\quad\Longleftrightarrow \quad
\bar\theta \geq C\rho^{\frac {2\alpha}{2-\alpha}}\,.
\]
There are then two possibilities: either $\bar\theta \leq C\rho^{\frac{2\alpha}{2-\alpha}}$, so we already have the validity of~(\ref{finlem}) and the proof is concluded, or $\bar\theta \geq C\rho^{\frac{2\alpha}{2-\alpha}}$, and then the last two estimates imply
\[
|V_f(E) - V_f(F)| \leq 22M\rho^2\bar\theta\,.
\]
Therefore, using Theorem~\ref{oldepsepsbeta} exactly as in the proof of Lemma~\ref{generic}, we find a set $\widetilde F$ having the same volume as $E$ (and then more perimeter) satisfying
\[
P_f(E) \leq P_f(\widetilde F) \leq P_f(F) + C \big(22M\rho^2 \bar\theta\big)^{\frac 1{2-\alpha}}
= P_f(F) + C \rho^{\frac 2{2-\alpha}}\bar\theta^{\frac 1{2-\alpha}} \,,
\]
from which we directly get
\begin{equation}\label{qufi}
\ell_f(\arc{xy})-\ell_f(xz)-\ell_f(zy) =P_f(E)-P_f(F)\leq C \rho^{\frac 2{2-\alpha}}\bar\theta^{\frac 1{2-\alpha}}\,.
\end{equation}
We now use the fact that~(\ref{2poss}) is valid both with $\bar\theta_1$ and with $\bar\theta_2$, as pointed out before. One has to distinguish three possible cases.\par

Since $\rho^{\frac \alpha{2-\alpha}}\leq \rho^{\frac \alpha{3-2\alpha}}$, if both $\bar\theta_1$ and $\bar\theta_2$ are smaller than $C\rho^{\frac \alpha{2-\alpha}}$ then so is $\bar\theta$, so~(\ref{finlem}) is true and there is nothing more to prove.\par

Suppose now that both $\bar\theta_1$ and $\bar\theta_2$ are bigger than $C\rho^{\frac \alpha{2-\alpha}}$. In this case, (\ref{2poss}) gives
\[
\ell_f(\arc{xy})-\ell_f(xz)-\ell_f(zy) =
\big(\ell_f(\arc{xz})-\ell_f(xz)\big) + \big(\ell_f(\arc{zy})-\ell_f(zy)\big)
\geq \frac \rho{12M} \,\big(\bar\theta_1^2+\bar\theta_2^2)
\geq \frac \rho{12M} \,\bar\theta^2\,,
\]
which together with~(\ref{qufi}) gives
\[
\rho\bar\theta^2 \leq C \rho^{\frac 2{2-\alpha}}\bar\theta^{\frac 1{2-\alpha}}\,,
\]
which is equivalent to~(\ref{finlem}), and then also in this case the proof is completed.\par

Finally, we have to consider what happens when only one between $\bar\theta_1$ and $\bar\theta_2$ is bigger than $C\rho^{\frac \alpha{2-\alpha}}$; just to fix the ideas, we can suppose that $\bar\theta_1\geq C\rho^{\frac \alpha{2-\alpha}}\geq \bar\theta_2$, hence in particular $\bar\theta=\bar\theta_1$. Applying then~(\ref{2poss}), this time we find
\[
\ell_f(\arc{xy})-\ell_f(xz)-\ell_f(zy) =
\big(\ell_f(\arc{xz})-\ell_f(xz)\big) + \big(\ell_f(\arc{zy})-\ell_f(zy)\big)
\geq \frac \rho{12M} \,\bar\theta^2 - C \rho^{\frac{2+\alpha}{2-\alpha}}
\geq \frac \rho{24M} \,\bar\theta^2\,,
\]
where the last inequality is true precisely because $\bar\theta\geq C \rho^{\frac \alpha{2-\alpha}}$. Exactly as before, putting this estimate together with~(\ref{qufi}) implies~(\ref{finlem}), and then also in the last possible case we obtained the proof.
\end{proof}

The last lemma is exactly what we needed to obtain the proof of our main Theorem~\mref{main}.

\begin{proof}[Proof (of Theorem~\mref{main})]
Let $E$ be an isoperimetric set for the ${\rm C}^{0,\alpha}$ density $f$. To show that $\partial E$ is of class ${\rm C}^{1,\frac \alpha{3-2\alpha}}$, let us select two generic points $z,\,a\in \partial E$ such that $\rho=\ell(za)\ll 1$. We want to show that
\begin{equation}\label{mainmain}
\angle wza \leq C \rho^{\frac \alpha{3-2\alpha}} \qquad \forall w \in \arc{az}\,,
\end{equation}
as this readily imply the thesis. Indeed, suppose that~(\ref{mainmain}) has been established, and call $\nu\in\S^1$ the direction of the segment $az$; since we already know that $\partial E$ is of class ${\rm C}^1$ by Theorem~\ref{old}, considering points $w\in \arc{az}$ which converge to $z$ we deduce by~(\ref{mainmain}) that $|\nu-\nu_z|\leq C \rho^{\frac \alpha{3-2\alpha}}$, where $\nu_z\in\S^1$ is the tangent vector of $\partial E$ at $z$. Since the situation of $a$ and of $z$ is perfectly symmetric, the same argument also shows that $|\nu-\nu_a|\leq C \rho^{\frac \alpha{3-2\alpha}}$, thus by triangual inequality we have found
\[
|\nu_a-\nu_z|\leq C \rho^{\frac \alpha{3-2\alpha}}\,.
\]
Since $a$ and $z$ are two generic points having distance $\rho$, and since $C$ only depends on $M$ and on $\alpha$, this estimate shows that $\partial E$ is of class ${\rm C}^{\frac \alpha{3-2\alpha}}$; therefore, the proof will be concluded once we show~(\ref{mainmain}). Notice that~(\ref{mainmain}) simply says that the estimate of Lemma~\ref{2.6} holds also without asking to the point $w$ to be closer to $a$ than to $z$.\par
To do so, let us recall that by~(\ref{esticirca}) it is $\ell(\arc{az}) \leq 2\rho$, let us write $a_0=a$, and let us define recursively the sequence $a_j$ by letting $a_{j+1}\in \arc {a_jz}$ be the point such that
\[
\ell(\arc{a_{j+1}z}) = \frac 23 \, \ell(\arc{a_jz})\,.
\]
Observe that $a_j$ is a sequence inside the curve $\arc{az}$, which converges to $z$, and moreover for every $j\in\N$ one has
\begin{equation}\label{stle}
\ell (a_j z) \leq \ell (\arc{a_j z}) = \bigg(\frac 23\bigg)^j \, \ell(\arc{az}) \leq 2\,\bigg(\frac 23 \bigg)^j\, \rho\,.
\end{equation}
Let us now take a point $w\in \arc{a_j a_{j+1}}$; again recalling~(\ref{esticirca}), by the definition of the points $a_j$ it is obvious that $w$ is closer to $a_j$ than to $z$; as a consequence, Lemma~\ref{2.6} applied to $a_j$ and $z$ ensures that
\[
\angle wz{a_j} \leq C\ell(a_jz)^{\frac \alpha{3-2\alpha}}
\leq C \, \kappa^j \rho^{\frac \alpha{3-2\alpha}} \qquad \forall \, w\in\arc{a_ja_{j+1}}\,,
\]
where we have also used~(\ref{stle}), and where $\kappa=(2/3)^{\frac \alpha{3-2\alpha}}<1$. Keeping in mind the obvious fact that $a_{j+1}\in \arc{a_ja_{j+1}}$ for every $j$, and then the above estimate is valid in particular for the point $a_{j+1}$, we deduce that for the generic $w\in \arc{a_j a_{j+1}}$ it is
\[
\angle wza = \angle wz{a_0} \leq \angle wz{a_j} + \sum_{i=0}^{j-1} \angle {a_{i+1}}z{a_i}
\leq C \rho^{\frac \alpha{3-2\alpha}} \sum_{i=0}^{j} \kappa^i
\leq C \rho^{\frac \alpha{3-2\alpha}}\,,
\]
where the last inequality is true because $\kappa=\kappa(\alpha)<1$. We have then established~(\ref{mainmain}), and then the proof is concluded.
\end{proof}


\begin{thebibliography}{100}
\bibitem{All1} W.K. Allard, A regularity theorem for the first variation of the area integrand, Bull. Amer. Math. Soc. \textbf{77} (1971), 772--776.
\bibitem{All2} W.K. Allard, On the first variation of a varifold. Ann. of Math. (2) 95 (1972), 417--491.
\bibitem{Alm} F.J. Almgren, Jr., Existence and regularity almost everywhere of solutions to elliptic variational problems with constraints, Mem. Amer. Math. Soc. \textbf{4} (1976), no. 165.
\bibitem {BigReg} F.J. Almgren, Jr., Almgren's big regularity paper. Q-valued functions minimizing Dirichlet's integral and the regularity of area-minimizing rectifiable currents up to codimension 2. World Scientific Monograph Series in Mathematics, Vol. \textbf{1}: xvi+955 pp. (2000).
\bibitem{Amb} L. Ambrosio, Corso introduttivo alla teoria geometrica della misura e alle superfici minime, Edizioni della Normale (1996).
\bibitem{AFP} L. Ambrosio, N. Fusco, D. Pallara, Functions of Bounded Variation and Free Discontinuity Problems, Oxford University Press (2000).
\bibitem{Bom} E. Bombieri, Regularity theory for almost minimal currents, Arch. Rational Mech. Anal. \textbf{78} (1982), no. 2, 99--130.
\bibitem{CMV} A. Ca\~nete, M. Jr. Miranda, D. Vittone, Some isoperimetric problems in planes with density, J. Geom. Anal. \textbf{20} (2010), no. 2, 243--290.
\bibitem{GC} G. Chambers, Proof of the Log-Convex Density Conjecture, preprint (2014). Available at http://arxiv.org/abs/1311.4012.
\bibitem{CP1} E. Cinti, A. Pratelli, The $\eps-\eps^\beta$ property, the boundedness of isoperimetric sets in $\R^N$ with density, and some applications, to appear on CRELLE (2015).
\bibitem{DavSem} G. David, S. Semmes, Quasiminimal surfaces of codimension 1 and John domains, Pacific J. Math. {\bf 183} (1998), no. 2, 213--277.
\bibitem{DFP} G. De Philippis, G. Franzina, A. Pratelli, Existence of Isoperimetric sets with densities ``converging from below'' on $\R^N$, preprint (2014).
\bibitem{DHHT} A. D\'iaz, N. Harman, S. Howe, D. Thompson, Isoperimetric problems in sectors with density, Adv. Geom. {\bf 12} (2012), no. 4, 589--619.
\bibitem{fusco} N. Fusco, The classical isoperimetric Theorem, Rend. Acc. Sc. Fis. Mat. Napoli, {\bf 71} (2004), 63--107.
\bibitem{FMP} N. Fusco, F. Maggi, A. Pratelli, The sharp quantitative isoperimetric inequality, Ann. of Math. (2) 168 (2008), no. 3, 941--980.
\bibitem{FMP2} N. Fusco, F. Maggi, A. Pratelli, On the isoperimetric problem with respect to a mixed Euclidean-Gaussian density, J. Funct. Anal. \textbf{260 }(2011), no. 12, 3678--3717.
\bibitem{GG} M. Giaquinta, E. Giusti, Quasi-minima, Ann. Inst. H. Poincar\'e Anal. Non Lin\'eaire {\bf 1} (1984), no. 2, 79--107.
\bibitem{KKLS} J. Kinnunen, R. Korte, A. Lorent, N. Shanmugalingam, Regularity of sets with quasiminimal boundary surfaces in metric spaces, J. Geom. Anal. {\bf 23} (2013), no. 4, 1607--1640.
\bibitem{Morgan2003} F. Morgan, Regularity of isoperimetric hypersurfaces in Riemannian manifolds, Trans. Amer. Math. Soc. {\bf 355} (2003), n. 12, 5041--5052.
\bibitem{Morgan2005} F. Morgan, Manifolds with density, Notices Amer. Math. Soc. \textbf{52} (2005), no. 8, 853--858. 
\bibitem{Morganbook} F. Morgan, Geometric Measure Theory: a Beginner's Guide, Academic Press, 4th edition, 2009.
\bibitem{MP} F. Morgan, A. Pratelli, Existence of isoperimetric regions in $\R^n$ with density, Ann. Global Anal. Geom. {\bf 43} n. 4 (2013), 331--365.
\bibitem{RCBM} C. Rosales, A. Ca\~nete, V. Bayle, F. Morgan, On the isoperimetric problem in Euclidean space with density, Calc. Var. Partial Differential Equations \textbf{31} (2008), no. 1, 27--46.
\bibitem{Tam} I. Tamanini, Regularity results for almost-minimal oriented hypersurfaces in $\R^N$, Quaderni del Dipartimento di Matematica dell'Universit\`a del Salento (1984).
\end{thebibliography}
\end{document}